\documentclass[12pt]{amsart}

\usepackage[top=1in, bottom=1in, left=1in, right=1in]{geometry}
\usepackage{times}

\usepackage{amssymb,amsmath,amsthm,mathtools,mathrsfs,bbm}
\usepackage{enumitem}

\setlist[1]{itemsep=0.5em, topsep=0.5em}

\usepackage[table,dvipsnames]{xcolor}
\usepackage{color}
\definecolor{red}{rgb}{1,0,0}
\definecolor{orange}{rgb}{0.7,0.3,0}
\definecolor{blue}{rgb}{0,.3,.7}
\definecolor{green}{rgb}{0,.6,.4}

\PassOptionsToPackage{hyphens}{url}\usepackage[colorlinks=true, linkcolor=NavyBlue, citecolor=teal, urlcolor=gray]{hyperref}   

\renewcommand{\le}{\leqslant}
\renewcommand{\leq}{\leqslant}
\renewcommand{\ge}{\geqslant}
\renewcommand{\geq}{\geqslant}{}

\numberwithin{equation}{section}

\theoremstyle{plain}
\newtheorem{theorem}{Theorem}
\newtheorem{cor}{Corollary}[section]
\newtheorem{corollary}[cor]{Corollary}

\newtheorem{lemma}[cor]{Lemma}

\newtheorem{proposition}[cor]{Proposition}

\theoremstyle{remark}

\newtheorem*{rem*}{Remark}
\newtheorem*{rems*}{Remarks}

\theoremstyle{definition}

\newtheorem{definition}{Definition}

\newcommand{\N}{\mathbb{N}}
\newcommand{\Z}{\mathbb{Z}}

\newcommand{\R}{\mathbb{R}}

\newcommand{\A}{\mathbf{A}}

\newcommand{\E}{\mathbb{E}}
\renewcommand{\P}{\mathbb{P}}

\newcommand{\CS}{\mathcal{S}}

\newcommand{\CU}{\mathcal{U}}

\newcommand{\CY}{\mathcal{Y}}

\newcommand{\Log}{\operatorname{Log}}

\newcommand{\eps}{\varepsilon}
\renewcommand{\phi}{\varphi}
\renewcommand{\rho}{\varrho}

\newcommand{\bs}\boldsymbol{}

\newcommand{\bg}{\big}
\newcommand{\bgg}{\Big}
\newcommand{\bggg}{\bigg}

\definecolor{red}{rgb}{1,0,0}

\definecolor{orange}{rgb}{0.7,0.3,0}

\definecolor{blue}{rgb}{.2,.6,.75}

\definecolor{green}{rgb}{.4,.7,.4}

\newcommand{\one}{\ensuremath{\mathbbm{1}}} 

\begin{document}

\title{A lower bound on the mean value of the Erd\H os--Hooley Delta function}

\author{Kevin Ford}
\address{Department of Mathematics \\
University of Illinois Urbana-Champaign\\
1409 W. Green Street \\
Urbana, IL 61801
USA}
\email{{\tt ford126@illinois.edu}}

\author{Dimitris Koukoulopoulos}
\address{D\'epartement de math\'ematiques et de statistique\\
Universit\'e de Montr\'eal\\
CP 6128 succ. Centre-Ville\\
Montr\'eal, QC H3C 3J7\\
Canada}
\email{{\tt dimitris.koukoulopoulos@umontreal.ca}}

\author{Terence Tao}
\address{Department of Mathematics\\
UCLA \\
405 Hilgard Ave\\
Los Angeles, CA 90095\\
USA}
\email{{\tt tao@math.ucla.edu}}


\date{\today}

\begin{abstract}
 We give an improved lower bound for the average of the Erd\H os--Hooley function 
 $\Delta(n)$, namely $\sum_{n\le x} \Delta(n) \gg_\varepsilon x(\log\log x)^{1+\eta-\varepsilon}$ for all $x\ge100$ and any fixed $\varepsilon$, where $\eta = 0.3533227\dots$ is an exponent previously appearing in work of Green and the first two authors.   This improves on a previous lower bound of $\gg x \log\log x$ of Hall and Tenenbaum, and can be compared to the recent upper bound of $x (\log\log x)^{11/4}$ of the second and third authors.
\end{abstract}

\maketitle

\section{Introduction}

The \emph{Erd\H os--Hooley Delta function} is defined for a natural number $n$ as
\[
\Delta(n) \coloneqq \max_{u \in \R} \#\{d|n: e^u<d\le e^{u+1}\} .
\]
Erd\H os introduced this function in the 1970s \cite{erdos1,erdos2} and studied certain aspects of its distribution in joint work with Nicolas \cite{erdos-nicolas,erdos-nicolas2}. However, it was not until the work of Hooley in 1979 that $\Delta$ was studied in more detail \cite{hooley}. Specifically, Hooley proved that
\begin{equation}
	\label{eq:hooley's estimate}
	\sum_{n\le x} \Delta(n) \ll x (\Log x)^{\frac{4}{\pi}-1} 
\end{equation}
for any $x\ge1$. Here and in the sequel we use the notation
\[
\Log x \coloneqq \max\{1, \log x\} \quad\text{for}\ x>0,
\]
and also define
\[
\Log_2 x \coloneqq \Log(\Log x); \quad
\Log_3 x \coloneqq \Log(\Log_2 x); \quad
\Log_4 x \coloneqq \Log(\Log_3 x).
\]
See also Section \ref{notation-sec} below for our asymptotic notation conventions.

Hooley's estimate \eqref{eq:hooley's estimate} has been improved by several authors \cite{HT1, HT2, HT3}, \cite{HT-book}, \cite{breteche}, \cite{KT}, culminating in the bounds
\begin{equation}
	\label{eq:HT}
	x\Log_2 x\ll \sum_{n\le x}\Delta(n) \ll x (\Log_2 x)^{11/4},
\end{equation}
with the lower bound established by Hall and Tenenbaum in \cite{HT1} (see also \cite[Theorem 60]{HT-book}), and the upper bound recently established in \cite{KT}. The main result of the present paper is an improvement of the lower bound in \eqref{eq:HT}. Our estimate is given in terms of the best known lower bounds for the normal order of $\Delta$, so we discuss these first. 

In \cite{breteche}, La Bret\`eche and Tenenbaum proved that
\[
\Delta(n) \le (\Log_2 x)^{\theta + \eps}
\]
for all fixed $\eps>0$ and all but $o(x)$ integers $n\in[1,x]$ as $x \to \infty$, where $\theta\coloneqq \frac{\log2}{\log2+1/\log2-1}=0.6102\dots$.  

To state the best known lower bound for the normal order, we need some additional notation, essentially from \cite{fgk}.

\begin{definition}\label{etastar-def}  If $A$ is a finite set of natural numbers, the \emph{subsum multiplicity} $m(A)$ of $A$ is defined
to be the largest number $m$ so that there area distinct subsets $A_1,\dots,A_m$ of $A$ such that
\begin{equation}\label{a-sum}
 \sum_{a \in A_1} a = \dots = \sum_{a \in A_m} a.
\end{equation}

One can think of the subsum multiplicity as a simplified model for the Erd\H{o}s--Hooley Delta function.

Now take $\A$ to be a random set of natural numbers, in which each natural number $a$ lies in $\A$ with an independent probability of $\P(a \in \A) = 1/a$.  If $k$ is a natural number, we define $\beta_k$ to be the supremum of all constants $c<1$ such that
\[
\lim_{D \to \infty} \P\Big( m\big( \A \cap [D^c, D] \big) \geq k \Big) = 1.
\]
It is shown in \cite{fgk}, by building on work in \cite{MT1}, that $\beta_k$ exists and is positive for all $k$.  We then define the quantity
\begin{equation}\label{eta-star}
\eta_* \coloneqq \liminf_{k \to \infty} \frac{\log k}{\log(1/\beta_k)},
\end{equation} 
thus $\eta_*$ is the largest exponent for which one has $\beta_k \geq k^{-1/\eta_*-o(1)}$ as $k \to \infty$.
\end{definition}

The main results of \cite{fgk} can then be summarized as follows:

\begin{theorem} \cite{fgk}
\begin{itemize}
\item[(i)]  We have $\eta_* \geq \eta$, where $\eta=0.353327\dots$ is defined by the formula
\begin{equation}\label{eta-def}
\eta \coloneqq \frac{\log 2}{\log(2/\rho)},
\end{equation}
and $\rho$ is the unique number in $(0,1/3)$ satisfying the equation $1 - \rho/2 = \lim_{j\to\infty} 2^{j-2}/\log a_j$    with $a_1 = 2$, $a_2 = 2 + 2^{\rho}$ and $a_j = a_{j-1}^2 + a_{j-1}^{\rho} - a_{j-2}^{2\rho}$ for $j\in\Z_{\ge3}$.
\item[(ii)] For any $\eps>0$, one has the lower bound
\begin{equation}\label{eta-eps}
\Delta(n) \geq (\Log_2 x)^{\eta_* - \eps}
\end{equation}
for all but $o(x)$ integers $n\in[1,x]$ as $x \to \infty$.
\end{itemize}
\end{theorem}

As a consequence of these results, we see that
\[
 (\Log_2 n)^{\eta - o(1)} \leq (\Log_2 n)^{\eta_* - o(1)} \leq \Delta(n) \leq (\Log_2 n)^{\theta + o(1)}
\]
as $n \to \infty$ outside of a set of zero natural density.  In particular $\eta \leq \eta_* \leq \theta$.  It is conjectured in \cite{fgk} that $\eta_* = \eta$ (and in fact $\beta_k = k^{-1/\eta-o(1)}$ as $k \to \infty$).

The main purpose of this note is to obtain an analogue of the lower bound in \eqref{eta-eps} for the mean value, thus improving the lower bound in \eqref{eq:HT}.

\begin{theorem}\label{main}  For any $\eps>0$ and all $x\ge 1$, we have
\[
\sum_{n\le x}\Delta(n)  \gg_\eps x (\Log_2 x)^{1+\eta_*-\eps}  \geq x (\Log_2 x)^{1+\eta-\eps}.
\]
\end{theorem}

Very informally, the idea of proof of the theorem is as follows.  Our task is to show that the mean value of $\Delta(n)$ is at least $(\Log_2 x)^{1+\eta_*-\eps}$.  In our arguments, it will be convenient to use a more ``logarithmic'' notion of mean in which $n$ is square-free and the prime factors of $n$ behave completely independently; see Section \ref{notation-sec} for details.  It turns out that for each natural number $r$ in the range
\begin{equation}
	\label{eq:r range flatness of Delta}
	(1+\eps) \Log_2 x \leq r \leq (2-\eps) \Log_2 x,
\end{equation}
there is a significant contribution (of size $\gg_\eps (\Log_2 x)^{\eta_*-\eps}$), arising from (squarefree) numbers $n$ whose number, $\omega(n)$, of prime factors is precisely $r$; summing in $r$ will recover the final factor of $\Log_2 x$ claimed. 

Suppose we write $r = (1+\alpha) \Log_2 x$ for some $\eps \leq \alpha \leq 1-\eps$ and we define $y$ such that.
\[
\Log_2 y = \alpha \Log_2 x + O(\sqrt{\Log_2 x}).
\]
It turns out that the dominant contribution to the mean from those numbers with $\omega(n)=r$ comes from those $n$ that factor as $n=n' n''$, where 
$n'$ is composed of primes $<y$, $n''$ is composed of primes $\ge y$,
$\omega(n') = 2\alpha \Log_2 x + O(\sqrt{\Log_2 x})$ and 
$\omega(n'') = (1-\alpha) \Log_2 x + O(\sqrt{\Log_2 x})$. 

 A pigeonholing argument (see Lemma \ref{pigeon}) then gives a lower bound roughly of the form
\[
\Delta(n) \gg (\Log_2 x)^{-o(1)} \max_y \frac{\tau(n')}{\log y} \Delta^*(n''),
\]
where $\Delta^*(m)$ is the maximum number of divisors of $m$ in an interval of the
form $(e^{u}, y e^u]$. The two factors $\frac{\tau(n')}{\log y}$ and $\Delta^*(n'')$ behave independently.  The arguments from \cite{fgk} will allow us to ensure that 
$\Delta^*(n'')\gg (\Log_2 x)^{\eta_* - \eps/2}$ with high probability, while the constraints on $n'$ basically allow us to assert that the $\frac{\tau(n')}{\log y}$ factor has bounded mean (after summing  over all the possible values of $\omega(n_{<y})$).  There is an unwanted loss of about $\frac{1}{\sqrt{\Log_2 x}}$ (related to the Erd\H{o}s--Kac theorem) coming from the restriction to $n$ having exactly $r$ prime factors, but this loss can be recovered by summing over the $\asymp \sqrt{\Log_2 x}$ essentially distinct possible values of $y$, after showing some approximate disjointness between events associated to different $y$ (see Lemma \ref{moment}).

\begin{rems*}
(a)  The above-described behavior that $\Delta$ seems to exhibit is rather unusual. For most arithmetic functions $f$, there is some constant $\varrho>0$ such that the dominant contribution to the partial sums $\sum_{n\le x}f(n)$ comes from integers $n$ with about $r_0\coloneqq \varrho\Log_2x$ prime factors. If we let $S_r(x)= \sum_{n\le x,\ \omega(n)=r}f(n)$, then we often have Gaussian-like decay as $r$ moves away from $r_0$, meaning that $S_r(x)\approx S_{r_0}(x) e^{-c(r-r_0)^2/\Log_2x}$ for some $c>0$ (e.g. when $f$ is the $k$-th divisor function). For some arithmetic functions, we have the even stronger exponential decay $S_r(x)\approx S_{r_0} e^{-c|r-r_0|}$ for some $c'>0$  (e.g. when $f$ is the indicator function of integers with a divisor in a given dyadic interval $[y,2y]$ - see \cite{ford:annals}). However, in the case of $f=\Delta$, the sums $S_r(x)$ appear to remain of roughly the same size for all $r$ in the range \eqref{eq:r range flatness of Delta};
  the lower bounds we obtain in Theorem \ref{main-2} below
  for $S_r(x)$ are of the same order for
  all $r$ in the range \eqref{eq:r range flatness of Delta}.
  
  \medskip
  
 (b) A few months after the publication of an arXiv preprint of the present paper, La Bret\`eche and Tenenbaum \cite{breteche-tenenbaum} proved the following result:
 \[
 	x(\Log_2 x)^{3/2} \ll \sum_{n\le x}\Delta(n) \ll x (\Log_2 x)^{5/2}\qquad(x\ge 1).
 \]
This improves Theorem \ref{main} and the upper bound in \eqref{eq:HT} that was proven in \cite{KT}. To obtain this improvement,  La Bret\`eche and Tenenbaum refined the methods of \cite{KT} and of the present paper.
\end{rems*}

\subsection*{Acknowledgments} 
KF is supported by National Science Foundation Grants DMS-1802139 and
DMS-2301264.

DK is supported by the Courtois Chair II in fundamental research, by the Natural Sciences and Engineering Research Council of Canada (RGPIN-2018-05699) and by the Fonds de recherche du Qu\'ebec - Nature et technologies (2022-PR-300951).

TT is supported by NSF grant DMS-1764034.

\section{Notation and basic estimates}\label{notation-sec}

We use $X \ll Y$, $Y \gg X$, or $X = O(Y)$ to denote a bound of the form $|X| \leq CY$ for a constant $C$.  If we need this constant to depend on parameters, we indicate this by subscripts, for instance $X \ll_k Y$ denotes a bound of the form $|X| \leq C_k Y$ where $C_k$ can depend on $k$.  We also write $X \asymp Y$ for $X \ll Y \ll X$.  All sums and products will be over natural numbers unless the variable is $p$, in which case the sum will be over primes. We use $\one\{E\}$ to denote the indicator of a statement $E$, thus $\one\{E\}$ equals $1$ when $E$ is true and $0$ otherwise. In addition, we write $\neg E$ for the negation of $E$. We use $\P$ for probability and $\E$ for probabilistic expectation. 

Given an integer $n$, we write $\tau(n)\coloneqq\sum_{d|n}1$ for its divisor-function and $\omega(n)\coloneqq\sum_{p|n}1$ for the number of its distinct prime factors.

It will be convenient to work with the following random model of squarefree integers.  For each prime $p$, let $n_p$ be a random variable equal to $1$ with probability $\frac{p}{p+1}$, and $p$ with probability $\frac{1}{p+1}$, independently in $p$.  Then for any $x \geq 1$, define the random natural number
\[
n_{<x} \coloneqq \prod_{p < x} n_p
\]
and similarly for any $1 \leq y < x$ define the natural number
\[
n_{[y,x)} \coloneqq \prod_{y \leq p < x} n_p.
\]
In particular we may factor $n_{<x}$ into independent factors $n_{<x} = n_{<y} n_{[y,x)}$ for any $1 \leq y < x$.
Observe that $n_{<x}$ takes values in the set  $\CS_{<x}$ denoting the set of square-free numbers, all of whose prime factors $p$ are such that $p<x$, with
\[ 
\P( n_{<x} = n ) = \frac{1}{n} \prod_{p<x} \bigg(1 + \frac{1}{p}\bigg)^{-1}
\]
for all $n \in \CS_{<x}$. In particular, from Mertens' theorem we have
\begin{equation}\label{efn}
 \E\big[ f(n_{<x})\big] \asymp \frac{1}{\Log x} \sum_{n \in \CS_{x}} \frac{f(n)}{n}
\end{equation}
for any non-negative function $f \colon \N \to \R^+$. 

We further note that
\begin{equation}\label{efn-2}
\begin{split}
 \E \big[ f(n_{<x}) \log n_{<x}\big] &= \sum_{p<x} \E \big[f(n_{<x}) \one\{p|n_{<x}\} \log p \big]  \\
&= \sum_{p<x} \E\bigg[ \frac{f(p n_{<x}) \log p}{p} \one\{p \nmid n_{<x}\} \bigg].
\end{split}
\end{equation}

We can also generalize \eqref{efn} to
\begin{equation}\label{efn-3}
 \E\big[ f(n_{[y,x)})\big] \asymp \frac{\Log y}{\Log x} \sum_{n \in \CS_{[y,x)}} \frac{f(n)}{n}
\end{equation}
where $\CS_{[y,x)}$ denotes the set of square-free numbers, all of whose prime factors 
lie in $[y,x)$.

We have the following elementary inequality:

\begin{proposition}  For any $x \geq 1$, we have
\[
\sum_{n \leq x} \Delta(n) \gg x \E \big[\Delta(n_{<x^{1/10}})\big] .
\]
\end{proposition}

\begin{proof} We may take $x$ to be sufficiently large.  Restricting attention to numbers $n \leq x$ of the form $n=mp$ where $m \leq \sqrt{x}/2$ and $\sqrt{x} < p \leq x/m$, we observe that $\Delta(n) \geq \Delta(m)$, and thus
\[
\sum_{n \leq x} \Delta(n) \geq \sum_{m \leq \sqrt{x}/2} \Delta(m) \sum_{\sqrt{x} < p \leq x/m} 1. 
\]
Hence, the Prime Number Theorem \cite[Theorem 8.1]{dk-book} implies that
\[
\sum_{n \leq x} \Delta(n) \gg \frac{x}{\log x} \sum_{m \leq \sqrt{x}/2} \frac{\Delta(m)}{m}.
\]
Restricting further to those $m$ in $\CS_{<y}$, $y=x^{1/10}$, we conclude from \eqref{efn} that
\begin{equation}\label{recip} 
	\begin{split}
 \sum_{n \leq x} \Delta(n) 
 	&\gg x\,  \E \bgg[ \Delta(n_{<y}) \one\bg\{n_{<y} \leq \sqrt{x}/2\bg\} \bgg]\\
 	&= x\bgg(  \E \big[ \Delta(n_{<y})\big] 
 	-  \E \bgg[ \Delta(n_{<y}) \one\bg\{n_{<y} > \sqrt{x}/2\bg\} \bgg] \bggg).
 		\end{split}
\end{equation}
Note from Markov's inequality and \eqref{efn-2} that
\begin{align*}
\E \bgg[\Delta(n_{<y}) \one\bg\{n_{<y} > \sqrt{x}/2 \bg\} \bgg] 
	&\leq \frac{1}{\log (\sqrt{x}/2)}	\E\big[\Delta(n_{<y}) \log n_{<y}\big] \\
&\leq \frac{2}{\log x - O(1)} \sum_{p<y} \E \bggg[\frac{\Delta(p n_{<y}) \log p}{p} \one\{p\nmid n_{<y}\} \bggg] .
\end{align*}
We have $\Delta(p n_{<y}) \leq 2 \Delta(n_{<y})$ whenever $p\nmid n_{<y}$. Thus, $\Delta(p n_{<y})\one\{p\nmid n_{<y}\}\le 2\Delta(n_{<y})$. Using this bound and Mertens' theorem \cite[Theorem 3.4(a)]{dk-book}, we conclude that
\[
\E \bgg[\Delta(n_{<y}) \one\bg\{n_{<y} > \sqrt{x}/2\bg\} \bgg]
	\leq \frac{4 \big(\log( x^{1/10})+O(1)\big)}{\log x - O(1)}
\E \big[ \Delta(n_{<y}) \big] \le \frac12 \, \E\big[\Delta(n_{<y})\big]
\]
for large enough $x$.  Combined with \eqref{recip}, this concludes the proof.
\end{proof}

Thus, to prove Theorem \ref{main}, it will suffice (after replacing $x$ with $x^{1/10}$) to establish the lower bound
\begin{equation}\label{ednx}
\E \bg[\Delta(n_{<x})\bg]  \gg_\eps (\Log_2 x)^{1+\eta_*-\eps}
\end{equation}
for all $\eps>0$, and $x$ sufficiently large in terms of $\eps$.  In fact we will show the following stronger estimate.

\begin{theorem}\label{main-2}  Let $\eps>0$ and $x>0$, and let $r$ be an integer in the range
\[
(1+\eps) \Log_2 x \leq r \leq (2-\eps) \Log_2 x.
\]
Then
\[
\E\bg[\Delta(n_{<x}) \one\{\omega(n_{<x})=r\} \bg] 
	\gg_\eps (\Log_2 x)^{\eta_*-\eps}.
\]
\end{theorem}

Clearly, Theorem \ref{main-2} implies \eqref{ednx} on summing 
over $r$.

We first record some basic information (cf.~\cite{norton}, \cite[Theorems 08, 09]{HT-book}) about the distribution of $\omega(n_{<x})$ (or more generally $\omega(n_{[y,x)})$), reminiscent of the Bennett inequality \cite{Bennett} but with a crucial additional square root gain in the denominator; it can also be thought of as a ``large deviations'' variant of the Erd\H{o}s--Kac law.

\begin{proposition}\label{prop:omnyx}  Let $1 \leq y \leq x$, with $x$ sufficiently large,
and let $k$ be a positive integer with
\[
k =t \cdot (\Log_2 x - \Log_2 y),
\]
where $t \le 10$. Then we have
\[
 \P\big(\omega(n_{[y,x)}) = k\big) \asymp \frac{\exp\big( - (\Log_2 x - \Log_2 y) Q(t) \big)}{\sqrt{k}},
\]
where $Q(t)=t\log t-t+1$.
\end{proposition}

\begin{proof} By \eqref{efn-3} we have
\[
\P\big(\omega(n_{[y,x)}) = k\big) \asymp \frac{\Log y}{\Log x}\, R_{y,x}, 
\quad\text{where}\quad R_{y,x} \coloneqq \sum_{y \leq p_1 < \dots < p_k < x} \frac{1}{p_1 \dots p_k}.
\]
Note that
\[
R_{y,x}\le \frac{(\sum_{y\le p<x}1/p)^k}{k!} 
	= \frac{(\Log_2x-\Log_2y+O(1/\Log y))^k}{k!}\ll \frac{(\Log_2x-\Log_2y)^k}{k!} 
\]
by Mertens' estimate \cite[Theorem 3.4(b)]{dk-book} and our assumption that $t\le 10$. Using the Stirling approximation $k! \asymp k^{1/2} (k/e)^k$, we obtain the claimed upper bound.

Lastly, we prove a corresponding lower bound.  Let $C$ be sufficiently large
and assume that $x \ge e^C$.
Set $y_1 = \max(y,C)$, and define
\[
L \coloneqq \sum_{y_1 \leq p < x} \frac{1}{p}.
\]
By Mertens' estimate, we have 
\begin{equation}
	\label{eq:L-mertens}
	L = \Log_2 x - \Log_2 y_1 + O(1/\Log y_1)
	\ge \Log_2x-\Log_2y_1-1/20,
\end{equation}
since $y_1\ge C$ and we may assume that $C$ is large enough. By hypothesis, $1\le k\le 10(\log_2 x-\Log_2 y)$, thus $\Log_2 x - \Log_2 y \ge 1/10$. If $y_1=y \ge C$, then we have $L \ge \frac12 (\Log_2 x-\Log_2 y)$.  Furthermore, if $y_1=C>y$, then $x\ge e^y$, whence $L \ge \Log_2 x - O(\Log_3 x) \ge \frac12 \Log_2 x$ provided that $C$ is large enough.  In both
cases,
\begin{equation}\label{L-lwr}
L \ge \frac12 (\Log_2 x-\Log_2 y),
\end{equation}
and it follows that $k\le 20L$.
In addition, we have
\begin{align*}
R_{y,x} &\ge R_{y_1,x} = \frac{L^k}{k!} 
- \frac{1}{k!} \sum_{y_1 \leq p_1,\dots,p_k < x \text{, not distinct}} \frac{1}{p_1 \dots p_k} \\
&\ge \frac{L^k}{k!} - \binom{k}{2} \frac{1}{k!} L^{k-2} \sum_{y_1 \leq p < x} \frac{1}{p^2}  \\
&\ge \frac{L^k}{k!} \bgg( 1 - \frac{k^2}{2L^2 (y_1-1)}  \bgg)
\ge \frac12 \cdot \frac{L^k}{k!},
\end{align*}
the last inequality holding for large enough $C$, since $y_1 \ge C$. By the upper bound on $t$ and inequality \eqref{L-lwr}, we have
\[
L^k \gg_C (L+\Log_2C + 1)^k .
\]
Together with \eqref{eq:L-mertens}, and since $\Log_2y_1\le \Log_2y+\Log_2C$, we conclude that
\[
L^k \gg_C (\Log_2 x-\Log_2 y)^k.
\]
Hence, the claimed lower bound on $R_{y,x}$ follows by Stirling's formula $k! \asymp k^{1/2} (k/e)^k$.
\end{proof}

We record two particular corollaries of the above proposition of interest, which follow from a routine Taylor expansion of the function $Q(t)$.  

\begin{corollary}[Special cases]\label{special}
Fix $B\ge 1$, let $1\le y\le x$ with $\Log_2 x - \Log_2 y \ge 2B^2$, and let $k\in\N$.  We have the two following cases:
\begin{itemize}
\item[(i)] If $\bg|k - (\Log_2 x - \Log_2 y)\bg| \le B \sqrt{\Log_2 x - \Log_2 y} $, then
\[
P(\omega(n_{[y,x)}) = k) \asymp_B \frac{1}{\sqrt{\Log_2 x - \Log_2 y}}.
\]
\item[(ii)] If $\bg|k - 2(\Log_2 x - \Log_2 y)\bg| \le B \sqrt{\Log_2 x - \Log_2 y}$, then
\[
P(\omega(n_{[y,x)}) = k) \asymp_B \frac{\Log x}{2^k \Log y} \cdot \frac{1}{\sqrt{\Log_2 x - \Log_2 y}}.
\]
\end{itemize}
\end{corollary}

\section{Main reduction}

Let the notation and hypotheses be as in Theorem \ref{main-2}.  We allow all implied constants to depend on $\eps$. We write
\[
r = (1+\alpha) \Log_2 x,
\]
thus 
\begin{equation}\label{alpha-range}
\eps \le \alpha \le 1-\eps.
\end{equation}
We may assume $x$ sufficiently large depending on $\eps$. For any $1 \leq y < x$, we have
\[
\omega(n_{<x}) = \omega(n_{<y}) + \omega(n_{[y,x)}).
\]
To take advantage of the splitting by $y$, we introduce a generalization
\begin{equation}\label{deltav-def}
\Delta^{(v)}(n) \coloneqq \max_{u \in \R} \#\{d|n: e^u<d\le e^{u+v}\} 
\end{equation}
of the Erd\H{o}s--Hooley Delta function for any $v>0$, and use the following simple application of the pigeonhole principle.

\begin{lemma}\label{pigeon}  For any $1 \leq y < x$ and any $v \geq \log n_{<y}$, we have
\[
\Delta(n_{<x}) \geq \frac{\tau(n_{<y})}{2v+1}  \cdot \Delta^{(v)}(n_{[y,x)}).
\]
\end{lemma}

\begin{proof}  By \eqref{deltav-def}, there exists $u$ such that there are 
$\Delta^{(v)}(n_{[y,x)})$ divisors $b$ of $n_{[y,x)}$ in $(e^u,e^{u+v}]$. 
Multiplying one of these divisors $b$ by any of the $\tau(n_{<y})$ divisors
$a$ of $n_{<y}$ gives a divisor $ab$ of $n_{<x}$ in the range $(e^u, e^{u+2v}]$. 
These $\tau(n_{<y})  \Delta^{(v)}(n_{[y,x)})$ divisors are all distinct.  
Covering this range by at most $2v+1$ intervals of the form 
$(e^{u'}, e^{u'+1}]$, we obtain the claim from the pigeonhole principle.
\end{proof}

Let 
\begin{equation}\label{y-range}
 \CY \coloneqq \bggg\{ y>0 : |\Log_2 y - \alpha \Log_2 x| \leq \sqrt{\Log_2 x}  \,; \quad \frac{\Log_2 y}{\log 2} \in \Z \bggg\}
\end{equation}
and for each $y \in \CY$, let $E_y$ denote the event
\[
\log n_{<y} \leq 10(\Log_3 x) \Log y
\]
and let $F_y$ denote the event
\begin{equation}\label{fy-def}
 \Delta^{(\Log y)}(n_{[y,x)}) \geq (\Log_2 x)^{\eta_* - \eps/2}.
\end{equation}
As we shall see later, both events $E_y$ and $F_y$ will hold with very high probability.
As $\Delta^{(v)}$ is clearly monotone in $v$, we have
\[
\Delta^{(10(\Log_3 x) \Log y)}(n_{[y,x)}) \geq \Delta^{(\Log y)}(n_{[y,x)}).
\]
By Lemma \ref{pigeon} with $v=10(\Log_3 x)\Log y$, if the events $E_y$ and $F_y$ both hold for some $y \in \CY$, then
\[
\Delta(n_{<x}) \gg \frac{(\Log_2 x)^{\eta_* - \eps/2}}{\Log_3 x} \cdot \frac{\tau(n_{<y})}{\Log y}.
\]
Thus Theorem \ref{main-2} will follow if we show that
\begin{equation}\label{eq:1st-reduction}
\E \bigg[  \one\bg\{ \omega(n_{<x})=r \bg\} \max_{y \in \CY} \bgg(\frac{\tau(n_{<y})}{\Log y} \one\{E_y\cap F_y\}\bgg)\bigg] \gg 1.
\end{equation}

Controlling the left-hand side is accomplished with the following three propositions. In their statements, recall that $\neg G$ denotes the negation of the event $G$.

\begin{proposition}\label{prop:Ey-fail} We have
\[
\E \bigg[   \one\bg\{ \omega(n_{<x})=r \bg\}  \max_{y \in \CY} \bgg( \frac{\tau(n_{<y})}{\Log y} \one\{\neg E_y\} \bgg) \bigg] \ll \frac{1}{(\Log_2 x)^{9}}.
\]
\end{proposition}

\begin{proof} We have 
\begin{align*}
\E \bigg[   \one\bg\{ \omega(n_{<x})=r \bg\}  \max_{y \in \CY} \bgg( \frac{\tau(n_{<y})}{\Log y} \one\{\neg E_y\} \bgg) \bigg] 
	&\le \sum_{y \in \CY}  \E \bigg[   \one\bg\{ \omega(n_{<x})=r \bg\}  \frac{\tau(n_{<y})}{\Log y} \one\{\neg E_y\}  \bigg] \\
	&\ll \sqrt{\Log_2 x} \, \max_{y \in \CY} \frac{ \E\big[ \tau(n_{<y}) \one\{\neg E_y\} \big] }{\Log y} ,
\end{align*}
since $|\CY| \ll (\Log_2 x)^{1/2}$. If $E_y$ fails, then
$n_{<y}^{1/\Log y} \ge \exp(10 \Log_3 x) = (\Log_2 x)^{10}$,
and so, by Markov's inequality,
\[
 \E\big[ \tau(n_{<y}) \one\{\neg E_y\} \big]
 	 \le (\Log_2 x)^{-10} \; \E\big[ \tau(n_{<y})  n_{<y}^{1/\Log y}\big]
\]
for all $y \in \CY$.  Splitting $n_{<y}$ into the independent factors $n_p$,
we get
\begin{equation}
	\label{eq:ub on tau}
		 \E \big[ \tau(n_{<y}) n_{<y}^{1/\Log y} \big] 
		= \prod_{p < y} \left( \frac{p}{p+1} + \frac{2 p^{1/\Log y}}{p+1} \right)
		= \prod_{p < y} \left( 1 + \frac{1}{p} + O\left( \frac{\log p}{p \Log y} + \frac{1}{p^2} \right) \right),
\end{equation}
which, by Mertens' theorems, equals $O(\Log y)$, and the proof is complete.
\end{proof}

\begin{proposition}\label{prop:supy}
We have 
\[
\E \bigg[ \one\{\omega(n_{<x})=r\}  \max_{y \in \CY} \frac{\tau(n_{<y})}{\Log y} \bigg] \gg 1.
\]
\end{proposition}

Proposition \ref{prop:supy} will be proved in Section \ref{sec:supy}.

\begin{proposition}\label{prop:Fy-fail}
We have  $\P\big(\neg F_y\big) \ll (\Log_2 x)^{-1}$.
\end{proposition}

Proposition \ref{prop:Fy-fail} will be proved in Section \ref{sec:prop-Fy}.

\medskip

Now we complete the proof of \eqref{eq:1st-reduction}, assuming the three propositions above. Firstly, by Proposition \ref{prop:Fy-fail} and the independence of $F_y$ and $n_{<y}$, we have
\begin{align*}
\E\bigg[ \one\{\omega(n_{<x})=r\} 	
	\max_{y \in \CY} \frac{\tau(n_{<y}) } {\Log y}\one\{\neg F_y\} \bigg]
&\ll (\Log_2 x)^{1/2} \max_{y \in \CY}   \frac{\E\big[\tau(n_{<y})\one\{\neg F_y\}\big]}{\Log y}  \\
&= (\Log_2 x)^{1/2} \max_{y \in \CY}   \frac{\E\big[\tau(n_{<y})\big]    \P\big(\neg F_y\big)}{\Log y}  \\
&\ll (\Log_2 x)^{-1/2} \max_{y \in \CY} \frac{\E\big[\tau(n_{<y})\big]}{\Log y}  .
\end{align*}
Arguing as in the proof of Proposition \ref{prop:Ey-fail}, we have $\E\big[\tau(n_{<y})\big]\ll \Log y$, and thus 
\[
\E\bigg[ \one\{\omega(n_{<x})=r\} 	
\max_{y \in \CY} \frac{\tau(n_{<y}) } {\Log y}\one\{\neg F_y\} \bigg]
	 \ll \frac{1}{(\Log_2 x)^{1/2}}. 
\]
Combining this with Propositions \ref{prop:Ey-fail} and \ref{prop:supy},
we deduce \eqref{eq:1st-reduction}, and hence Theorem \ref{main-2}.

\section{Proof of Proposition \ref{prop:supy}}\label{sec:supy}
Let $y\in \CY$, so that $\Log_2 y / \log 2$ is an integer.
Consider the events
\[
E_{y,u}\coloneqq \bggg\{ \omega(n_{<x})=r,\  \omega(n_{<y}) =  \frac{\Log_2 y}{\log 2} + u\bggg\}.
\]
In the event $E_{y,u}$ we have $\tau(n_{<y})=2^u\Log y$. Also, consider the events
\[
G_u\coloneqq \bigcup_{y\in \CY} E_{y,u}, \qquad H_u = \bigcup_{u'\ge u} G_u.
\] 
 Thus,
\begin{align*}
\E \bigg[\one\{\omega(n_{<x})=r\} \max_{y\in \CY}  \frac{\tau(n_{<y})}{\Log y} \bigg]
&= \E \bigg[ \max_{y\in \CY} \sum_{u\in \Z} 2^u \cdot \one\{E_{y,u}\} \bigg] \\
&= \sum_{u\in \Z} 2^u \P \bg( H_u\setminus  H_{u+1} \bg)\\
&= \sum_{u\in \Z} 2^u \bg( \P(H_u) - \P(H_{u+1}) \bg)\\
&= \sum_{u\in \Z} 2^{u-1} \P(H_u)\\
&\ge  \sum_{u\in \Z} 2^{u-1} \P(G_u).
\end{align*}
In fact, $\sum_u 2^{u-1} \P(H_u) \le \sum_u 2^u G_u$, so we have lost at most a factor
$1/2$ in the final step.
We will restrict attention to the most important values of $u$, namely $u\in \CU$,
where
\[
\CU \coloneqq \bggg\{ u\in \Z: \bgg| u - \frac{\log 4 - 1}{\log 2} \alpha \Log_2 x \bgg| \le \sqrt{\Log_2 x} \bggg\}.
\]
This choice is informed by the calculations in \cite[Proposition 4.1]{KT}. We then observe that
\begin{equation}\label{Gu}
\E \bigg[\one\{\omega(n_{<x})=r\} \max_{y\in \CY}  \frac{\tau(n_{<y})}{\Log y} \bigg]
\ge \sum_{u\in \CU} 2^{u-1} \P(G_u) .
\end{equation}
It remains to bound $\P(G_u)$ from below for $u\in\CU$. This follows essentially by more general results of Ford \cite{ford:smirnov and primes}, but we may give a simple and self-contained argument in the special case we are interested in. To do so, we employ the second moment method.  More precisely, we have the following estimates:

\begin{lemma}\label{moment}  We have
\begin{equation}\label{eyu}
 \P( E_{y,u} ) \asymp \frac{2^{-u}}{\Log_2 x} \qquad (u\in \CU, y\in \CY)
\end{equation}
 and, for all $y,y'\in \CY$ and $u\in \CU$,
\begin{equation}\label{eyu-2}
 \P\bg( E_{y,u} \cap E_{y',u} \bg) \ll \frac{2^{-u}}{\Log_2 x} \exp\bgg( - Q(1/\log 2)\bg|\Log_2 y-\Log_2 y'\bg| \bgg) .
\end{equation}
\end{lemma}

Before we prove Lemma \ref{moment}, let us see how to use it to bound $\P(G_u)$ from below. 

\medskip

For any given $u\in \CU$, Lemma \ref{moment} yields that
\[
\E\bigg[ \sum_{y \in \CY} \one\{E_{y,u}\} \bigg] 
	\asymp \frac{2^{-u}}{\sqrt{\Log_2 x}}
\]
and
\[
\E \bigg[ \bgg(\sum_{y \in \CY} \one\{E_{y,u}\}\bgg)^2 \bigg]
\ll \frac{2^{-u}}{\sqrt{\Log_2 x}} .
\]
On the other hand, the Cauchy--Schwarz inequality implies that
\[
\E\bigg[ \sum_{y \in \CY} \one\{E_{y,u}\} \bigg]^2 \le 
	\P(G_u) \cdot \E \bigg[ \bgg(\sum_{y \in \CY} \one\{E_{y,u}\}\bgg)^2 \bigg]
\]
It follows that
\[
\P(G_u) \gg \frac{2^{-u}}{\sqrt{\Log_2 x}}. 
\]
for any $u\in \CU$. Inserting this into \eqref{Gu} completes the proof of Proposition \ref{prop:Ey-fail}.

\begin{proof}[Proof of Lemma \ref{moment}]
  We begin with \eqref{eyu}.  Splitting 
$ \omega(n_{<x}) = \omega(n_{<y}) + \omega(n_{[y,x)})$
and using the independence of $n_{<y}$ and $n_{[y,x)}$, we may factor
\begin{align*}
 \P(E_{y,u}) = 
 	\P\bigg( \omega(n_{<y}) =  \frac{\Log_2 y}{\log 2} + u \bigg) 
\cdot \P\bigg( \omega(n_{[y,x)}) = r - \frac{\Log_2 y}{\log 2} - u \bigg).
\end{align*}
From the range of $y$ and $u$, we have
\[
\frac{\Log_2 y}{\log 2} + u  = 2 \alpha \Log_2 x + O\left( \sqrt{\Log_2 x} \right) = 2 \Log_2 y + O\left(\sqrt{\Log_2 y}\right)
\]
and
\begin{align*}
 r - \frac{\Log_2 y}{\log 2}- u  &= (1-\alpha) \Log_2 x + O\left( \sqrt{\Log_2 x} \right) \\
 &= \Log_2 x - \Log_2 y + O\left(\sqrt{\Log_2 x - \Log_2 y}\right).
 \end{align*}
We may thus invoke parts (i), (ii) of Corollary \ref{special} and conclude that
\[
\P\bigg( \omega(n_{<y}) = \frac{\Log_2 y}{\log 2}+ u \bigg)  \asymp \frac{2^{-u}}{\sqrt{\Log_2 x}}
\]
and
\[
\P\bigg( \omega(n_{[y,x)}) = r -  \frac{\Log_2 y}{\log 2}- u \bigg) \asymp \frac{1}{\sqrt{\Log_2 x}},
\]
and \eqref{eyu} follows.

Now we establish \eqref{eyu-2}.  Without loss of generality, we may assume that $y'>y$.  Splitting
\[
\omega(n_{<x}) = \omega(n_{<y}) + \omega(n_{[y,y')}) + \omega(n_{[y',x)})
\]
and using the joint independence of $n_{<y}, n_{[y,y')}, n_{[y',x)}$, we may factor
\begin{align*}
 \P\bg( E_{y,u} \cap E_{y',u} \bg)
 	 &= \P\bigg( \omega(n_{<y}) = \frac{\Log_2 y}{\log 2}+ u \bigg) \\
	&\quad \times \P\bigg( \omega(n_{[y,y')}) = \frac{\Log_2 y'}{\log 2}- \frac{\Log_2 y}{\log 2}\bigg)\\
	&\quad \times \P\bigg( \omega(n_{[y',x)}) = r -  \frac{\Log_2 y'}{\log 2}- u \bigg).
\end{align*}
As before, we have
\[
\P\bigg( \omega(n_{<y}) = \frac{\Log_2 y}{\log 2}+ u \bigg)  \asymp \frac{2^{-u}}{\sqrt{\Log_2 x}}
\]
and
\[
\P\bigg( \omega(n_{[y',x)}) = r -  \frac{\Log_2 y'}{\log 2}- u \bigg) \asymp \frac{1}{\sqrt{\Log_2 x}}
\]
while from Proposition \ref{prop:omnyx}  we also have
\[
\P\bigg( \omega(n_{[y,y')}) = \frac{\Log_2 y' - \Log_2 y}{\log 2} \bigg)
\ll 
	\exp\bgg( - Q(1/\log 2) \bg(\Log_2 y' - \Log_2 y\bg) \bgg),
\]
and the claim \eqref{eyu-2} follows.
\end{proof}

\section{Proof of Proposition \ref{prop:Fy-fail}}\label{sec:prop-Fy}

Fix $\eps>0$; we allow all implied constants to depend on $\eps$.  We will also need the parameters $k\in \N$, sufficiently large in terms of $\eps$, and $\delta>0$, sufficiently
small in terms of $\eps$ and $k$.
We may assume that $x$ is sufficiently large depending on $\eps, k,\delta$.
We have 
\begin{align*}
 \P\bgg( \omega( n_{[y,x)} ) \geq 2 \Log_2 x \bgg)  &\le \frac{\E \big[\tau(n_{[y,x)})\big]}{(\Log x)^{\log 4}} \\
 &= \frac{1}{(\Log x)^{\log 4}} \prod_{y \le p<x} \frac{p+2}{p+1} \ll \frac{1}{(\Log x)^{\log 4-1}},
 \end{align*}
 using Mertens' theorem.
By the definition \eqref{fy-def} of $F_y$, it will thus suffice to show that
\begin{equation}\label{dest}
\P\bgg( \Delta^{(\Log y)}(n_{[y,x)}) < (\Log_2 x)^{\eta_* - \eps/2}, \ \omega(n_{[y,x)}) < 2 \Log_2 x \bgg) \ll (\Log_2 x)^{-1}.
\end{equation}
It is now convenient to replace the random integer $n_{[y,x)}$ with a more discretized model.  We introduce the scale
\[
\lambda \coloneqq y^{\frac{1}{20 \Log_2 x}}
\]
(note that this is large depending on $\eps,k,\delta$, since $\Log_2 y \asymp \Log_2 x$ by \eqref{alpha-range} and \eqref{y-range}) and let $J$ be the set of all integers $a$ such that 
\[
y \leq \lambda^{a-2} < \lambda^{a-1} \leq x.
\]
In particular we have
\[
a \geq 20 \Log_2 x + 2,
\]
so that $a$ is sufficiently large depending on $\eps,k,\delta$.

Define the random subset $\A$ of $J$ to consist of all the indices $a \in J$ for which one has $n_p = p$ for some prime $p\in[\lambda^{a-2},\lambda^{a-1})$.  Observe that the events $\{a \in \A\}$ are mutually independent, with
\begin{align*}
 \P( a \in\A )
 	&= 1 - \prod_{\lambda^{a-2} \leq p < \lambda^{a-1}} \frac{p}{p+1} \\
&= 1 - \frac{a-2}{a-1}\bgg(1+ O\bgg(e^{-c\sqrt{\log y}}\bgg) \bgg)  \\
	&= \frac{1}{a-1} + O\bgg(e^{-c\sqrt{\log y}}\bgg) 
\end{align*}
for large enough $y$, thanks to the Prime Number Theorem \cite[Theorem 8.1]{dk-book} and the bound $\lambda^{a-2} \ge y$. Note that $a\le\Log x \ll e^{(\Log y)^{1/3}}$. Hence, if $x$ is large enough, we find that
\[
 \P( a \in\A ) \ge \frac{1}{a}\quad\text{for all}\ a\in J .
\]

Recall Definition \ref{etastar-def} of the subsum multiplicity $m(A)$ of a finite set $A$ of natural numbers. We claim that if $\omega(n_{[y,x)}) < 2\Log_2 x$, then
\begin{equation}\label{mu-bound}
 \Delta^{(\Log y)}(n_{[y,x)}) \geq m(\A).
\end{equation}
Indeed, if $\mu=m(\A)$, then Definition \ref{etastar-def} implies that we can find distinct subsets $A_1,\dots,A_\mu$ of $\A$ such that
\[
\sum_{a \in A_1} a = \dots = \sum_{a \in A_\mu} a.
\]
Also, for each $a \in \A$, we can find a prime $p_a|n_{[y,x)}$ such that
\begin{equation}\label{lam2}
\lambda^{a-2} \leq p_a < \lambda^{a-1}.
\end{equation}
In particular, the cardinality of $\A$ (and hence of $A_1,\dots,A_\mu$) is at most $\omega(n_{[y,x)})$, and hence at most $2 \Log_2 x$.  Taking logarithms in \eqref{lam2}, we see that 
\[
|\log p_a - a \log \lambda| \leq 2 \log \lambda
\]
for all $a \in \A$. Thus for all $1 \leq j < j' \leq \mu$ we have from the triangle inequality, \eqref{a-sum}, and the bound $|\A|\le2\Log_2x$ that
\[
\bgg|\sum_{a \in A_j} \log p_a - \sum_{a \in A_{j'}} \log p_a\bgg| \leq 8 (\Log_2 x) \log \lambda < \log y
\]
thanks to the choice of $\lambda$.  We conclude that all the sums $\sum_{a \in A_j} \log p_a$, $j=1,\dots,\mu$ lie in an interval of the form $(u, u+\log y]$, hence all the products $\prod_{a \in A_j} p_a$ lie in an interval of the form $(e^u, e^{u+\log y}]$.  As these products are all distinct factors of $n_{[y,x)}$, the claim \eqref{mu-bound} follows.

\medskip

In view of \eqref{mu-bound}, it now suffices to establish the bound
\[
\P\Big( m(\A) < (\Log_2 x)^{\eta_* - \eps/2}  \Big) \ll (\Log_2 x)^{-1} .
\]
The events $a \in \A$ for $a \in J$ are independent with a probability of at least $1/a$.  One can then find a random subset $\A'$ of $\A$ where the events $a \in \A'$ for $a \in J$ are independent with a probability of \emph{exactly} $1/a$; for instance, one could randomly eliminate each $a \in \A$ from $\A'$ with an independent probability of $1 - \frac{1}{a \P(a \in \A)}$.  Clearly $m(\A) \geq m(\A')$, so it will suffice to show that
\begin{equation}\label{muap}
 \P\bgg( m(\A') < (\Log_2 x)^{\eta_* - \eps/2}  \bgg) \ll (\Log_2 x)^{-1}.
\end{equation}

Now we use a ``tensor power trick'' going back to the work of Maier and Tenenbaum \cite{MT1} (see also \cite[Lemma 2.1]{fgk}).  Observe the supermultiplicativity inequality
\begin{equation}\label{mua12}
 m(A_1 \cup A_2) \geq m(A_1) m(A_2)
\end{equation}
whenever $A_1, A_2$ are disjoint finite sets of natural numbers.  To exploit this, we introduce the exponent $0 < c < 1$ by the formula
\begin{equation}\label{c-def}
 c^{\eta_*-\eps/4} = 1/k,
\end{equation}
where $k$ was defined at the start of this section.  By construction, the set $J$ takes the form $[a_-, a_+]\cap\Z$, where
\[
a_- = \frac{\log y}{\log \lambda} + O(1) \asymp \Log_2 x
\]
and
\[
a_+ = \frac{\log x}{\log \lambda} + O(1) \gg \Log^{\eps/2} x
\]
thanks to \eqref{y-range} and \eqref{alpha-range}.  In particular,
\[
\Log_2 a_+ - \Log_2 a_- \geq \Log_3 x - O(\Log_4 x).
\]
As a consequence of this and \eqref{c-def}, we can find an integer $\ell$ satisfying
\begin{equation}\label{m-int}
\ell = \frac{\Log_3 x}{\log(1/c)} - O(\Log_4 x) = (\eta_* - \eps/4) \frac{\Log_3 x}{\log k} - O(\Log_4 x)
\end{equation}
and disjoint sets $J_1,\dots,J_\ell$ in $J$, where each $J_i$ is of the form $J_i = [D_i^c, D_i]\cap\Z$ for some $D_i \gg \Log_2 x.$ From \eqref{mua12} and monotonicity, we then have
\begin{equation}\label{mua-lower}
 m(\A') \geq \prod_{i=1}^\ell m(\A' \cap J_i).
\end{equation}

From the definition of $\eta_*$ in \eqref{eta-star}, we have $c < \beta_k$ if $k$ is large enough.  From the definition of $\beta_k$ in Definition \ref{etastar-def}, we conclude (as $D_i$ is sufficiently large depending on $k, \delta$) that 
\[
\P\bgg( m( \A' \cap J_i ) \geq k \bgg) \geq 1 - \delta
\]
for all $i=1,\dots,\ell$.  Furthermore, the events $m(\A' \cap J_i) \geq k$ are independent, because the sets $\A' \cap J_i$ are independent.  By the Bennett inequality \cite{Bennett}, the probability that there are fewer than $(1-\eps/5)\ell$ values of $i$ with $m(\A'\cap J_i)\ge k$ is at most
\[
\exp \bggg\{ - \ell\cdot (\eps/5-\delta) \bgg( \log\bgg(1+ \frac{\eps/5-\delta}{\delta}\bgg)-1 \bgg) \bggg\}.
\]
We choose $\delta$ small enough so that
\[
 (\eps/5-\delta) \bgg( \log\bgg(1+ \frac{\eps/5-\delta}{\delta}\bgg)-1 \bgg) \ge
 \frac{2\log k}{\eta_* - \eps/4},
\]
so that by \eqref{m-int}, the above probability is $\ll (\Log_2 x)^{-3/2}$. By \eqref{mua-lower}, we conclude that
\[
\P( m(\A') < k^{(1-\frac{\eps}{5}) \ell} ) \ll (\Log_2 x)^{-3/2}.
\]
 From another appeal to \eqref{m-int} we have
\[
k^{(1-\frac{\eps}{5}) \ell} \geq (\Log_2 x)^{\eta_* - \eps/2}.
\]
The claim \eqref{muap} follows.


\begin{thebibliography}{alpha}

\bibitem{Bennett}
G.~Bennett. \emph{Probability inequalities for the sum of independent random variables.} 
Amer. Stat. Assoc. J. {\bf 57} (1962), 33--45.

\bibitem{breteche}
R.~de la Bret\`eche, G.~Tenenbaum, 
\emph{Two upper bounds for the Erd\H{o}s--Hooley Delta-function.} 
Sci. China Math. 66 (2023), no. 12, 2683--2692.

\bibitem{breteche-tenenbaum}
\bysame,\,
{\it Note on the mean value of the Erd\H os--Hooley Delta-function.}
Preprint (2024), arXiv:2309.03958, 12 pages.




\bibitem{erdos1}
P.~Erd\H{o}s, \emph{Problem 218.} 
Can. Math. Bull. \textbf{16} (1973), pp.~463. 

\bibitem{erdos2}
\bysame, \emph{Problem 218, Solution by the proposer.} 
Can. Math. Bull. \textbf{17} (1974), 621--622.

\bibitem{erdos-nicolas}
P.~Erd\H{o}s, J.-L.~Nicolas, 
\emph{R\'epartition des nombres superabondants.} 
Bull. Soc. math. France \textbf{103} (1975), 65--90.

\bibitem{erdos-nicolas2}
\bysame, 
\emph{M\'ethodes probabilistes et combinatoires en th\'eorie des nombres.}
Bull. Sci. Math. (2), 100 (1976), pp. 301--320.

\bibitem{ford:smirnov and primes}
K.~Ford,\,
{\it Generalized Smirnov statistics and the distribution of prime factors.}
Funct. Approx. Comment. Math. 37 (2007), part 1, 119--129.

\bibitem{ford:annals}
\bysame,\,
{\it The distribution of integers with a divisor in a given interval.} 
    Ann.~of Math.~(2) 168 (2008), 367--433.

\bibitem{fgk}
K.~Ford, B.~Green and D.~Koukoulopoulos, 
\emph{Equal sums in random sets and the concentration of divisors.}
Invent.~Math. 232 (2023), no.~3, 1027--1160.

\bibitem{HT1}
R.~R.~Hall and G.~Tenenbaum, G. \emph{On the average and normal orders of Hooley's $\Delta$-function.} 
J.~London Math.~Soc. (2) 25 (1982), no. 3, 392--406.

\bibitem{HT2}
\bysame, 
\emph{The average orders of Hooley's $\Delta_r$-functions.} 
Mathematika 31 (1984), no. 1, 98--109.

\bibitem{HT3}
\bysame, 
\emph{The average orders of Hooley's $\Delta_r$-functions. II.} 
Compositio Math.~60 (1986), no. 2, 163--186.

\bibitem{HT-book}
\bysame, \emph{Divisors.}
Cambridge Tracts in Mathematics, 90. Cambridge University Press, Cambridge, 1988.

\bibitem{hooley}
C.~Hooley, \emph{On a new technique and its applications to the theory of numbers.} 
Proc. London Math. Soc. (3) \textbf{38} (1979), no. 1, 115--151.

\bibitem{dk-book}
D.~Koukoulopoulos, 
\emph{The distribution of prime numbers.} 
Graduate Studies in Mathematics, 203. American Mathematical Society, Providence, RI, 2019.

\bibitem{KT}
D.~Koukoulopoulos and T. Tao, 
\emph{An upper bound on the mean value of the Erd\H os--Hooley Delta function.}
Proc. Lond. Math. Soc. (3) 127 (2023), no. 6, 1865–1885.

\bibitem{MT1}
H.~Maier and G.~Tenenbaum,\,
{\it On the set of divisors of an integer.} 
Invent. Math. 76 (1984), no. 1, 121--128. 

\bibitem{MT2}
\bysame
{\it On the normal concentration of divisors.} 
J. London Math. Soc. (2) 31 (1985), no. 3, 393--400.

\bibitem{MT3}
\bysame
{\it On the normal concentration of divisors. II.}
Math. Proc. Cambridge Philos. Soc. 147 (2009), no. 3, 513--540.

\bibitem{norton}
K.~K.~Norton,
\emph{On the number of restricted prime factors of an integer. I.}
Illinois J. Math. 20 (1976), no. 4, 681--705.

\bibitem{tenenbaum}
G.~Tenenbaum, 
\emph{Sur la concentration moyenne des diviseurs.} 
Comment. Math. Helv. \textbf{60} (1985), no. 3, 411--428.

\end{thebibliography}
\end{document}